\documentclass[12pt]{amsart} 
\usepackage{amsmath,amssymb,amscd,amsthm,fullpage}
\usepackage{latexsym}
\usepackage{colortbl}

\newtheorem{theorem}{Theorem}

\newtheorem{remark}{Remark}

\newtheorem{corollary}{Corollary}

\renewcommand{\[}{\left[}
\renewcommand{\(}{\left(}
\renewcommand{\]}{\right]}
\renewcommand{\)}{\right)}
\newcommand{\conv}{\mathop{\textup{conv}}}
\renewcommand{\le}{\leqslant}
\renewcommand{\ge}{\geqslant}

\newcommand{\ga}{\gamma}

\newcommand{\la}{\lambda}

\newcommand{\R}{\mathbb R}              

\begin{document}
\hfill\today
\bigskip
\author{Jaegil Kim, Vladyslav Yaskin and Artem Zvavitch}

\address{Department of Mathematics, Kent State University,
Kent, OH 44242, USA} \email{jkim@math.kent.edu}


\dedicatory{Dedicated to the memory of Nigel J. Kalton, 1946-2010.}
\address{ Department of Mathematical and
Statistical Sciences,
University of Alberta,
Edmonton, AB, T6G 2G1,
Canada } \email{vladyaskin@math.ualberta.ca}

\address{Department of Mathematics, Kent State University,
Kent, OH 44242, USA} \email{zvavitch@math.kent.edu}

\thanks{The first and third authors were supported in part by U.S.~National Science Foundation grant
 DMS-0652684. The second author was supported in part by NSERC} 
\title{The geometry of $p$-convex intersection bodies}

\begin{abstract} Busemann's theorem states that the intersection body of an origin-symmetric convex body is also convex.
In this paper we provide a version of Busemann's theorem for $p$-convex bodies. We show that the intersection body of a $p$-convex body is $q$-convex for certain $q$. Furthermore, we discuss the sharpness of the previous result  by constructing an appropriate example. This example is also used to show that $IK$, the intersection body of $K$, can be much farther away from the Euclidean ball than $K$.  Finally, we extend these theorems to some general measure spaces with log-concave and $s$-concave measures. \end{abstract}

\maketitle

\section{Introduction and Notations}

A {\em body} is   a compact  set with nonempty interior. For a body $K$ which is star-shaped with respect to the origin its {\em radial function}  is defined by
$$
\rho_K(u) = \max \{\la \ge 0 : \la u \in K\} \quad\mbox{for every } u\in S^{n-1}.
$$
A body $K$ is called a {\em star body} if it is star-shaped at the origin and its radial function $\rho_K$ is positive and continuous.
The {\em Minkowski functional} $\|\cdot\|_K$   is defined by $\|x\|_K = \min \{\la \ge 0 : x \in \la K\}$. Clearly, $\rho_K(u)=\|u\|_K^{-1}$, for $u \in S^{n-1}$. Moreover, we can assume that this identity holds for all $u\in \mathbb R^n\setminus\{0\}$ by extending $\rho_K$ from the sphere to $\mathbb R^n\setminus\{0\}$ as a homogeneous function of degree $-1$.

In \cite{Lu1}, Lutwak introduced the notion of the {\em intersection body} $IK$ of a star body $K$. $IK$ is defined by its radial function $$\rho_{IK}(u)= |K\cap u^\bot| , \quad \mbox{ for   } u \in S^{n-1}.$$
Here and throughout the paper,   $u^\perp$ denotes the hyperplane perpendicular to $u$, i.e.
$ u^\perp=\{x\in \R^n: x\cdot u=0\}$.  By $|A|_k$, or simply $|A|$ when there is no ambiguity, we denote the $k$-dimensional Lebesgue measure of a set $A$.

We refer the reader to the books \cite{Ga2}, \cite{Ko}, \cite{KoY} for more information on the definition and properties of intersection bodies, and their applications in convex geometry and geometric tomography.

In this paper we are interested in the properties of the operator $I$ that assigns to a star body $K$ its intersection body $IK$.
One of the well-known results is the classical Busemann theorem   (\cite{Bu}, see also \cite{Ba1}, \cite{Ba2} and  \cite{MP}).

\begin{theorem}\label{th:Busemann} {
Let $K$ be an origin-symmetric convex body in $\R^n$. Then its intersection body $IK$ is convex.
}
\end{theorem}

Recently, a new  proof of Busemann's theorem
was established by Berck \cite{Be},  who also generalized the theorem to the case of  $L_p$ intersection bodies (see  \cite{GG}, \cite{HL},  \cite{YY} and  \cite{Ko} for more information on the theory of $L_p$ intersection bodies).

The development of the theory of intersection bodies shows that it is not natural to restrict ourselves to the class of convex bodies, and in fact, in many situations one has to deal with bodies which are not necessarily convex. How does $I$ act on these bodies?   In this paper we will answer this question for the class of $p$-convex bodies.

 Let $p\in (0,1]$. We say that a  body $K$ is
$p$-convex if, for all $x, y \in \R^n$,
$$
\|x+y\|^p_K \le \|x\|^p_K+\|y\|^p_K,
$$
or, equivalently $t^{1/p}x+(1-t)^{1/p}y \in K$ whenever $x$ and $y$
are in $K$ and $t\in (0,1)$. One can see that
 $p$-convex sets with $p=1$ are just convex. Note also that a $p_1$-convex body is $ p_2$-convex for all $0<p_2\le p_1$.
There is an extensive literature on $p$-convex bodies as well as the closely related concept of quasi-convex bodies, see for example \cite{A}, \cite{BBP1}, \cite{BBP2}, \cite{D}, \cite{GoK},  \cite{GoL}, \cite{GuL}, \cite{Ka}, \cite{KaT}, \cite{KaPR}, \cite{Li1}, \cite{Li2}, \cite{Li3}, \cite{LMP}, \cite{LMS}, \cite{M} and others.

The first question that we consider is the following.
 {\it Let $K$ be an origin-symmetric $p$-convex body. Is the intersection body $IK$   necessarily $q$-convex for some $q$?
 }

In Section 2, we prove that if $K$ is $p$-convex and symmetric then its intersection body $IK$ is $q$-convex for all  $q \le \left[
\big(\frac 1p-1\big)(n-1) + 1 \right]^{-1}$. Furthermore, we construct an example showing that this upper bound   is asymptotically sharp.

Another important question about the operator $I$ comes from works of  Lutwak \cite{Lu2} and Gardner \cite[Prob. 8.6, 8.7]{Ga2} (see also \cite{GrZ}). It is easy to see that the intersection body of a Euclidean ball is again a Euclidean ball. A natural question is whether there are other fixed points of $I$. In order to measure the distance between star  bodies  we will be using the Banach-Mazur distance
$d_{BM}(K,L)=\inf \{b/a: \exists T \in GL(n): aK \subset TL \subset bK \}$ (see \cite{MS}).
In \cite{FNRZ} it is shown that in a sufficiently small neighborhood of the ball (with respect to the Banach-Mazur distance) there are no other fixed points of the  intersection body operator. However, in general this question is still open.
In view of this it is natural to ask   the following question.

{\it Does $IK$ have to be closer to the ball than $K$?  }

In Section 2 we show that the answer is ``No". There are $p$-convex bodies for which the intersection body is farther from the Euclidean ball.

It is worth noting that, in the convex situation,  there exists an absolute constant $C>0$ such that
$d_{BM}(IK,B_2^n)<C$ for all origin-symmetric convex bodies $K$  (see \cite{He}, \cite{Ba1},  \cite{Ba2}, \cite{Bou}, \cite{MP} or Corollary 2.7 in \cite{Ko}). The example in Section 2 shows that this statement is wrong if we only assume $p$-convexity for $p<1$, and  $d_{BM}(IK,B_2^n)$, for a $p$-convex body $K$, can be as large as   $C_p^n$, where $C_p>1$ is a certain constant that depends only on $p$.

In recent times a lot of attention has been attracted to the study of log-concave measures. These are measures whose densities are log-concave functions. The interest to such measures stems from the Brunn-Minkowski inequality, and many results known for convex bodies are now generalized to log-concave measures,  see for example \cite{Ba1}, \cite{Ba2},  \cite{AKM}, \cite{KlM}, \cite{FM}, \cite{Pa} and the references cited therein.

In Section 3 we study  intersection bodies in  spaces with log-concave measures. Namely, let $\mu$ be a log-concave measure on $\mathbb R^n$ and $\mu_{n-1}$ its restrictions to $(n-1)$-dimensional subspaces.  Define the {\em intersection body} $I_\mu K$ of a star body $K$ {\em with respect to $\mu$} by
$$
\rho_{I_\mu K}(u) = \mu_{n-1}(K\cap u^\bot),\quad u \in S^{n-1}.
$$
We show that if $K$ is an origin-symmetric $p$-convex body  and   $\mu$ is a symmetric and log-concave measure, then $I_\mu K$ is $q$-convex for all $q \le \left[
\big(\frac 1p-1\big)(n-1) + 1 \right]^{-1}$. The proof uses a version of Ball's theorem \cite{Ba1}, \cite{Ba2} for $p$-convex bodies. Namely, we show that if $f$ is an even non-negative log-concave function on $\mathbb R^n$, $k\ge 1$, and $K$  is a $p$-convex body in $\R^n$, $0<p\le 1$, then the body $L$ defined by the Minkowski functional
$
\|x\|_L = \left[\int_0^{\|x\|_K^{-1}} f(rx)r^{k-1}dr \right]^{-1/k}, $ $x\in\R^n, $
is $p$-convex.

If the measure $\mu$ is not symmetric, the situation is different, as explained at the end of Section 3.  $L$ defined above does not have to be $q$-convex for any $q>0$. However, if we consider $s$-concave measures,  $0<s<1/n$, that are not necessarily symmetric, then the above construction defines a body $L$ that is $q$-convex for all $q\le \[\big(\frac 1p-1\big)\big(\frac 1s-n\big)\frac 1k + \frac 1p\]^{-1}$. We also show that this bound is sharp. This is the content of Section 4.


\noindent {\bf Acknowledgment}. We are indebted to Alexander Litvak and Dmitry Ryabogin for
valuable discussions.

\section{$p$-convexity and related results}

Here we prove  a version of Busemann's theorem for $p$-convex bodies. 

\begin{theorem}\label{th:pbus}

Let $K$ be a symmetric $p$-convex body in $\R^n$,   $p \in (0,1]$, and $E$ a $(k-1)$-dimensional subspace of $\R^n$ for $1\le k\le n$. Then the map
$$
u \longmapsto \frac{|u|}{\big|K \cap {\rm span}(u,E)\big|_k},\quad u\in E^\perp
$$
defines the Minkowski functional of a $q$-convex body in $E^\bot$ with $q =\left[
\left(1/p-1\right)k + 1 \right]^{-1}$.
\end{theorem}

\begin{proof} We follow the general idea of the proof from \cite{MP}  (see also \cite[p.311]{Ga2}).
Let $u_1, u_2 \in  E^\perp \setminus \{0\}$ be nonparallel vectors. Denote $u=u_1+u_2$,   and
\begin{eqnarray*}
\rho (u_1) &=&\frac{ |K \cap \mbox{span}\{u_1,E\}|}{|u_1|} = \int_{-\infty}^{\infty}
\left|K \cap (r u_1+E)\right|dr, \\
\rho (u_2) &=&\frac{ |K \cap \mbox{span}\{u_2,E\}| }{|u_2|}= \int_{-\infty}^{\infty} \left|K
\cap (r u_2+E)\right|dr.
\end{eqnarray*}
Define the functions $r_1 = r_1(t)$ and $r_2 = r_2(t)$ by
\begin{eqnarray*}
t &=& \frac{1}{\rho (u_1)} \int_0^{r_1} \left|K \cap(r
u_1+E)\right|dr \\
&=& \frac{1}{\rho (u_2)} \int_0^{r_2} \left|K \cap(r
u_2+E)\right|dr ,\qquad t\in [0,1/2].
\end{eqnarray*}
Let $r = \left( r_1^{-p} + r_2^{-p} \right)^{-\frac{1}{p}}$, $\la_1
= \frac{r_1^{-p}}{r_1^{-p} + r_2^{-p}}$ and $\la_2 =
\frac{r_2^{-p}}{r_1^{-p} + r_2^{-p}}$. Then

\begin{eqnarray*}
\frac{dr}{dt} &=& -\frac{1}{p}\left(r_1^{-p} +
r_2^{-p}\right)^{-\frac{1}{p}-1} \left( -p\,
r_1^{-p-1}\frac{dr_1}{dt}
- p\, r_2^{-p-1}\frac{dr_2}{dt} \right) \\
&=& r \left[ \la_1 \left(\frac{1}{r_1}\frac{dr_1}{dt}\right) +
\la_2 \left(\frac{1}{r_2}\frac{dr_2}{dt}\right) \right] \\
&\ge& r \left(\frac{1}{r_1}\frac{dr_1}{dt}\right)^{\la_1}
\left(\frac{1}{r_2}\frac{dr_2}{dt}\right)^{\la_2} =
\left(\la_1^{1/p}\frac{dr_1}{dt}\right)^{\la_1}
\left(\la_2^{1/p}\frac{dr_2}{dt}\right)^{\la_2} \\
&=& \left(\la_1^{\la_1}\la_2^{\la_2}\right)^{\frac{1}{p}}
\left(\frac{\rho (u_1)}{|K \cap (r_1 u_1+E)|}\right)^{\la_1}
\left(\frac{\rho (u_2)}{|K \cap (r_2 u_2+E)|}\right)^{\la_2}.
\end{eqnarray*}
On the other hand, since $r\,u = (r_1^{-p}+r_2^{-p})^{-\frac{1}{p}}
(u_1+u_2) = \la_1^{\frac{1}{p}}r_1u_1+\la_2^{\frac{1}{p}}r_2u_2$, we
have
\begin{eqnarray*}
K \cap (r\,u+E) &\supset& \la_1^{\frac{1}{p}}\left(K
\cap (r_1\, u_1+E)\right) + \la_2^{\frac{1}{p}}\left(K \cap (r_2\,u_2+E)\right) \\
&=& \la_1\left(\la_1^{\frac{1}{p}-1}K \cap (r_1\, u_1+E)\right) +
\la_2\left(\la_2^{\frac{1}{p}-1}K \cap (r_2\,u_2+E)\right).
\end{eqnarray*}
Thus, by the Brunn-Minkowski  inequality (see,  for example,  \cite{Ga3}), we get
\begin{eqnarray*}
|K \cap (ru+E)| &\ge& \left|\la_1^{\frac{1}{p}-1}K \cap (r_1
u_1+E)\right|^{\la_1} \left|\la_2^{\frac{1}{p}-1}K \cap (r_2
u_2+E)\right|^{\la_2} \\
&=& \left(\la_1^{\la_1}\la_2^{\la_2}\right)^{(\frac{1}{p}-1)(k-1)}
\left|K \cap (r_1 u_1+E)\right|^{\la_1} \left|K \cap (r_2
u_2+E)\right|^{\la_2}.
\end{eqnarray*}
Finally, we have the following:
\begin{eqnarray*}
\rho (u_1+u_2) &=& \int_{-\infty}^{\infty} \left|K \cap (r
u+E)\right|dr = 2\int_0^{1/2} \left|K \cap (r
u+E)\right|\frac{dr}{dt}dt \\
&\ge& 2\int_0^{1/2}
\left(\la_1^{\la_1}\la_2^{\la_2}\right)^{(\frac{1}{p}-1)(k-1)+\frac{1}{p}}
\rho (u_1)^{\la_1}\rho (u_2)^{\la_2} dt \\
&\ge& 2\int_0^{1/2}
 \left[ \Big(\la_1
\left[\rho (u_1)\right]^q\Big)^{\la_1} \Big(\la_2
\left[\rho (u_2)\right]^q\Big)^{\la_2} \right]^{1/q} dt\\
&\ge& 2\int_0^{1/2} \left[ \frac{\la_1}{\la_1
\left[\rho (u_1)\right]^q} + \frac{\la_2}{\la_2
\left[\rho (u_2)\right]^q}
\right]^{-1/q} dt\\
&=& \left[ \left[\rho (u_1)\right]^{-q} +
\left[\rho (u_2)\right]^{-q} \right]^{-1/q}.
\end{eqnarray*}
Therefore $\rho(u)$ defines a $q$-convex body.

\end{proof}

As an immediate corollary of the previous theorem we obtain the following.

\begin{theorem}\label{th:q-conv}
Let $K$ be a symmetric $p$-convex body in $\R^n$ for $p \in (0,1]$. Then the
intersection body $IK$ of $K$ is $q$-convex for   $q = \left[
\left(1/p-1\right)(n-1) + 1 \right]^{-1}$.
\end{theorem}
\begin{proof}
Let $L=IK$ be the intersection body of $K$. Let $v_1,v_2\in \mbox{span}\{u_1,u_2\}$ be orthogonal to $u_1$ and $u_2$ correspondingly. Denote $E=\mbox{span}\{u_1,u_2\}^\perp$.
Then $$\rho_L(v_1) =  |K \cap \mbox{span}\{u_1,E\}| = \rho (u_1),$$
and $$\rho_L(v_2) =  |K \cap \mbox{span}\{u_2,E\}| = \rho (u_2).$$
Using the previous theorem with $k=n-1$, we see that $L$ is $q$-convex for   $q =  [
\left(1/p-1\right)(n-1) + 1 ]^{-1}$.
\end{proof}

\begin{remark}
Note that the previous theorem does not hold without the symmetry assumption. To see this, use the idea from \cite[Thm 8.1.8]{Ga2}, where it is shown that $IK$ is not necessarily convex if $K$ is not symmetric.
\end{remark}

A natural question is to see whether the value of $q$ in Theorem \ref{th:q-conv} is optimal. Unfortunately, we were unable to construct a body that gives exactly this value of $q$, but our next result shows that the bound is asymptotically correct.

\begin{theorem}

There exists a $p$-convex body $K\subset \mathbb R^n$ such that $IK$ is $q$-convex with $q\le \left[(1/p-1)(n-1)+1 +g_n(p)\right]^{-1}$, where $g_n(p)$ is  a function that satisfies

1) $g_n(p)\ge  - \log_2(n-1),$

2) $\displaystyle \lim_{p\to 1^-} g_n(p) =0 .$

\end{theorem}

\begin{proof}

Consider the following two $(n-1)$-dimensional cubes in $\mathbb R^{n}$:
$$C_1=\{|x_1|\le 1, ...,|x_{n-1}|\le 1, x_{n}=1\} \,\,\,\,\,\,  \mbox{  and  } \,\,\,\,\,\, C_{-1}=\{|x_1|\le 1, ...,|x_{n-1}|\le 1, x_{n}=-1\}.$$
For a fixed $0<p<1$, let us define  a set $K\subset \mathbb R^{n}$ as follows:
$$
K = \{z\in \mathbb R^n :  z= t^{1/p}x+ (1-t)^{1/p}y, \mbox{ for some }  x\in C_{1}, y\in C_{-1},  0\le t\le 1\}.
$$
We claim that $K$ is $p$-convex. To show this let us consider two arbitrary points $z_1, z_2\in K$,
$$z_1= t_1^{1/p}x_1+ (1-t_1)^{1/p}y_1,\qquad z_2= t_2^{1/p}x_2+ (1-t_2)^{1/p}y_2,$$ where $x_1, x_2\in C_{1}$, $y_1, y_2\in C_{-1}$, and $t_1, t_2 \in [0, 1]$.

We need to show that for all $s\in (0,1)$ the point  $w= s^{1/p}z_1+ (1-s)^{1/p}z_2$ belongs to $K$.


Assume first that $t_1$ and $t_2$ are  neither both   equal to   zero  nor both equal to one.
Since $C_1$ and $C_{-1}$ are convex sets, it follows that the points
$${\bar x} = \frac{s^{1/p}t_1^{1/p}x_1+ (1-s)^{1/p}t_2^{1/p}x_2}{s^{1/p}t_1^{1/p} + (1-s)^{1/p}t_2^{1/p}}
\,\,\,\,\,\,\mbox{  and   }\,\,\,\,\,\,
{\bar y} =\frac{s^{1/p} (1-t_1)^{1/p}y_1+  (1-s)^{1/p}(1-t_2)^{1/p}y_2}{s^{1/p} (1-t_1)^{1/p}+  (1-s)^{1/p}(1-t_2)^{1/p}}$$
belong to $C_1$ and $C_{-1}$ correspondingly. Then $w=  \alpha {\bar x} + \beta {\bar y}$,  where $$\alpha = s^{1/p}t_1^{1/p} + (1-s)^{1/p}t_2^{1/p} \,\,\,\,\,\,  \mbox {   and }\,\,\,\,\,\,   \beta =s^{1/p} (1-t_1)^{1/p}+  (1-s)^{1/p}(1-t_2)^{1/p}.
$$
Note that $\alpha^p +\beta^p \le s t_1  + (1-s) t_2 + s  (1-t_1) +  (1-s) (1-t_2) = 1.$
Therefore, there exists $\mu\ge 0$ such that $(\alpha+\mu)^p+(\beta+\mu)^p =1$ and
 $$w=  (\alpha+\mu) {\bar x} + (\beta {\bar y}+\mu (-{\bar x})) = (\alpha+\mu) {\bar x} + (\beta+\mu)\frac{\beta {\bar y}+\mu (-{\bar x})}{\beta +\mu}.$$
 Since ${\bar y}\in C_{-1}$ and $-{\bar x}\in C_{-1}$, it follows that $${\widetilde y} = \frac{\beta {\bar y}+\mu (-{\bar x})}{\beta +\mu}\in C_{-1}.$$

Therefore $w$ is a $p$-convex combination
 of ${\bar x}\in C_{1}$ and ${\widetilde y}\in C_{-1}$.

 If $t_1$ and $t_2$ are either both zero or one, then either $\alpha=0$ or $\beta=0$. Without loss of generality let us say $\alpha=0$, then $w = \beta {\bar y}.$ Now choose $\bar x\in C_1$ arbitrarily, and apply the considerations above to the point $w=  0 {\bar x} + \beta {\bar y}.$
 The claim follows.


Note that $K$ can be written as
\begin{eqnarray} \label{lab:K}
K &=& \left\{ r^\frac1p x+ (1-r)^\frac1p y : x\in C_{1}, y\in C_{-1},  0\le r\le 1 \right\}  \nonumber \\
&=& \left\{ r^\frac1p v+ (1-r)^\frac1p w + \[r^\frac1p -(1-r)^\frac1p \]e_n: v, w\in B_\infty^{n-1},  0\le r\le 1 \right\}  \nonumber \\
&=& \left\{ \[r^\frac1p +(1-r)^\frac1p \] z + \[r^\frac1p -(1-r)^\frac1p \]e_n : z\in B_\infty^{n-1}, 0\le r\le 1\right\} \nonumber \\
&=& \Big\{ f(t) z + te_n :  z \in B_\infty^{n-1},   -1\le t\le 1\Big\} ,
\end{eqnarray}
where  $B^{n-1}_\infty=[-1,1]^{n-1}\subset \R^{n-1}$ and $f$ is a function on $[-1,1]$ defined as the solution $s=f(t)$ of $$\(\frac{s+t}{2}\)^p + \(\frac{s-t}{2}\)^p =1, \quad s\ge|t|,-1\le t\le 1.$$
Let $L=IK$ be the intersection body of $K$. Then
$$\rho_L(e_n)= |K\cap e_n^\perp| = \left(2f(0)\right)^{n-1}=\left(\frac{4}{2^{1/p}}\right)^{n-1} .$$

In order to compute the volume of the central section of $K$ orthogonal to $(e_1+e_n)/\sqrt{2}$, use  (\ref{lab:K}) to notice that its projection onto $x_1=0$ coincides with $K\cap e_1^\perp$. Therefore
$$\rho_L((e_1+e_n)/\sqrt{2} )  = \sqrt{2}\rho_L( e_1 )=2 \sqrt{2}\int_0^1 \left[2f(t)\right]^{n-2} dt .$$

Let  $L=IK$ be $q$-convex. In order to estimate $q$, we will  use the inequality
$$\left\|\sqrt{2} e_n\right\|_L^q\le \left\|\frac{e_n+e_1}{\sqrt{2}}\right\|_L^q+ \left\|\frac{e_n-e_1}{\sqrt{2}}\right\|_L^q = 2 \left\|\frac{e_n+e_1}{\sqrt{2}}\right\|_L^q,$$ that is
$$\frac{\sqrt{2}}{\rho_L(e_n)} \le \frac{2^{1/q}}{\rho_L((e_1+e_n)/\sqrt{2} )} .$$
Thus, we have 
$$2^{1/q}\ge \left(\frac{2^{1/p}}{2}\right)^{n-1} 2  { \int_0^1 f(t)^{n-2} dt }  . $$
We now  estimate the latter integral.
From the definition of $f$ it follows that $f(t)\ge \frac{2}{2^{1/p}}$ and $f(t)\ge  t $ for $t\in [0,1]$. Taking the maximum of these two functions, we get   $$\displaystyle f(t) \ge \left\{\begin{array}{ll}\frac{2}{2^{1/p}},& \mbox{ for } 0\le t \le \frac{2}{2^{1/p}},\\
t,& \mbox{ for }   \frac{2}{2^{1/p}} \le t\le 1. \end{array} \right.$$
Therefore, \begin{align}\label{intf} \int_0^1 f(t)^{n-2} dt&\ge\int_0^{\frac{2}{2^{1/p}}}\left(\frac{2}{2^{1/p}}\right)^{n-2} dt +  \int_{\frac{2}{2^{1/p}}}^1 t^{n-2} dt\\\notag &=\left(\frac{2}{2^{1/p}}\right)^{n-1}+\frac{1}{n-1}\left(1 - \left(\frac{2}{2^{1/p}}\right)^{n-1} \right).\end{align}
 Hence,
 $$2^{1/q}\ge \ 2    \left(1+\frac{1}{n-1}\left(\left(\frac{2^{1/p}}{2}\right)^{n-1} - 1\right)  \right) , $$
which implies $$q\le \left[\left(\frac1p-1\right)(n-1)+1+\log_2\frac{(n-2)\left(\frac{2}{2^{1/p}}\right)^{n-1}+1}{n-1}\right]^{-1}.$$
Denoting $$g_n(p)=\log_2\frac{(n-2)\left(\frac{2}{2^{1/p}}\right)^{n-1}+1}{n-1},$$
we get the statement of the theorem.




\end{proof}


We will use the above example to show that in general  the intersection body operator does not improve the Banach-Mazur distance to the Euclidean ball $B_2^n$.

\begin{theorem}
Let $p\in (0,1)$ and let $c$ be any constant satisfying $1<c< {2^{1/p-1}}$.  Then for all large enough $n$, there exists a $p$-convex body $K\subset \mathbb R^n$ such that
$$c^n d_{BM}(K, B_2^n)  <  d_{BM} (IK, B_2^n).$$

\end{theorem}
\begin{proof}  We will consider $K$ from the previous theorem.  One can see that $K \subset B_\infty ^n \subset \sqrt{n}B_2^n$. Also note that for any $a \in B_\infty^n$, there exist $x \in C_1$,  $y\in C_{-1}$ and $\lambda \in [0,1]$ such that $a=\lambda x + (1-\lambda)y$. Then we have
$
\|a\|^p_K \le \lambda^p+(1-\lambda)^p\le 2^{1-p}.
$
Therefore, $K \supset 2^{\frac{p-1}{p}}B_\infty^n \supset  2^{\frac{p-1}{p}}B_2^n$, and thus
\begin{equation}\label{KB}
d_{BM}(K, B_2^n) \le  2^{\frac{1-p}{p}} \sqrt{n}.
\end{equation}
Next we would like to provide a lower bound for  $d_{BM}(IK, B_2^n)$.  Let
$E$ be an ellipsoid such that
$
E \subset IK \subset d E,
$ for some $d$.
Then
$$
IK \subset \conv(IK) \subset d E  \,\,\,\,\,\,\  \mbox{  and   }\,\,\,\,\,\,\, \frac{1}{d}\conv(IK) \subset E \subset IK.
$$
Therefore,  $1/d \le 1/r$, where $r=\min \{ t:  conv(IK) \subset  t IK\}$.   Thus, $$d_{BM}(IK, B_2^n) \ge r=\max\left\{\frac{\rho_{conv(IK)}(\theta)}{\rho_{IK}(\theta)},  \theta \in S^{n-1}\right\} \ge  \frac{\rho_{conv(IK)}(e_n)}{\rho_{IK}(e_n)}.$$
The  convexity of $conv(IK)$ gives
$$
\rho_{conv(IK)}(e_n) \ge  \left\| \frac{1}{2} \left  (\rho_{IK}\left(\frac{e_n+e_1}{\sqrt{2}}\right) \frac{e_n+e_1}{\sqrt{2}}+ \rho_{IK}\left(\frac{e_n-e_1}{\sqrt{2}}\right)\frac{e_n-e_1}{\sqrt{2}}\right)\right\|_2$$
$$
= \frac{1}{\sqrt{2}} \rho_{IK}\left(\frac{e_n+e_1}{\sqrt{2}}\right) .
$$
Combining  the above inequalities   with  inequality (\ref{intf}) from the previous theorem, we get
$$
d_{BM}(IK, B_2^n) \ge \frac{\rho_{IK}(\frac{e_n+e_1}{\sqrt{2}})} {\sqrt{2}\rho_{IK}(e_n)} \ge \left(\frac{2^{1/p}}{2}\right)^{n-1} \frac{1}{n-1}.    $$  Comparing this with (\ref{KB}) we get the statement of the theorem.
\end{proof}

\section{Generalization to log-concave measures}

A measure $\mu$ on $\mathbb R^n$ is called log-concave if for any measurable $A,B \subset \mathbb R^n$
and $0 < \lambda < 1$, we have
$$
\mu(\lambda A + (1-\lambda)B) \ge  \mu(A)^\lambda \mu(B)^{(1-\lambda)}
$$
whenever $\lambda A + (1-\lambda)B$ is measurable.

Borell \cite{Bor} has shown that a measure $\mu$ on $\mathbb R^n$ whose support is not contained
in any affine hyperplane is a log-concave measure if and only if it is
absolutely continuous with respect to the Lebesgue measure, and its
density is a log-concave function.

To extend Busemann's theorem to log-concave measures on $\R^n$, we need the following theorem of Ball \cite{Ba1}, \cite{Ba2}.

\begin{theorem} \label{th:ball}
Let $f:\R^n\rightarrow [0,\infty)$ be an even log-concave function satisfying $0<\int_{\mathbb R^n} f <\infty$ and let $k\ge 1$. Then the map
$$
x \longmapsto \[\int_0^\infty f(rx)r^{k-1}dr\]^{-\frac 1k}
$$
defines a norm on $\R^n$.
\end{theorem}

An immediate consequence of Ball's theorem is a generalization of the classical Busemann theorem to log-concave measures on $\R^n$.

 Let $\mu$ be a measure on $\R^n$, absolutely continuous  with respect to the Lebesgue measure $m$, and $f$   its density function.
 If $f$ is locally integrable on $k$-dimensional affine subspaces of $\R^n$, then we denote by $\mu_k = f m_k$ the restriction of $\mu$ to $k$-dimensional subspaces, where $m_k$ is the $k$-dimensional Lebesgue measure.

Define the {\em intersection body} $I_\mu K$ of a star body $K$ {\em with respect to $\mu$} by
$$
\rho_{I_\mu K}(u) = \mu_{n-1}(K\cap u^\bot),\quad u \in S^{n-1}.
$$

Let $\mu$ be a symmetric  log-concave measure on $\R^n$ and $K$ a symmetric convex body in $\R^n$.
Let $f$ be the density of the measure $\mu$. If we apply Theorem \ref{th:ball} to the log-concave function $1_K f$, we get a symmetric convex body $L$ whose Minkowski functional is given by $$\|x\|_L = \[(n-1)\int_0^\infty (1_K f)(rx)r^{n-2}dr\]^{-\frac{1}{n-1}}.$$ Then for every $u \in S^{n-1}$,
\begin{eqnarray*}
\mu_{n-1}(K\cap u^\perp) &=& \int_{S^{n-1}\cap u^\perp}\int_0^\infty (1_K f)(r\theta)r^{n-2} dr d\theta \\
&=& \frac{1}{n-1}  \int_{S^{n-1}\cap u^\perp} \|\theta\|_L^{-n+1} d\theta = |L\cap u^\perp|.
\end{eqnarray*}
Using   Theorem \ref{th:Busemann} for the convex body $L$,  one immediately obtains the following version of Busemann's theorem for log-concave measures.

\begin{theorem}
Let $\mu$ be a symmetric  log-concave measure on $\R^n$ and $K$ a symmetric convex body in $\R^n$. Then the intersection body $I_\mu K$ is convex.
\end{theorem}

In order to generalize Theorem \ref{th:q-conv} to log-concave measures, we will first prove a version of Ball's theorem (Thm \ref{th:ball}) for $p$-convex bodies.

\begin{theorem}\label{th:log-concave}
Let $f:\R^n\rightarrow [0,\infty)$ be an even log-concave function, $k\ge 1$, and $K$ a $p$-convex body in $\R^n$ for $0<p\le 1$. Then the body $L$ defined by the Minkowski functional
$$
\|x\|_L = \left[\int_0^{\|x\|_K^{-1}} f(rx)r^{k-1}dr \right]^{-\frac 1k}, \quad x\in\R^n,
$$
is $p$-convex.
\end{theorem}

\begin{proof}
Fix two non-parallel vectors $x_1,x_2\in\R^n$ and denote $x_3=x_1+x_2$. We claim that $\|x_3\|_L^p \le \|x_1\|_L^p + \|x_2\|_L^p$. Consider the following 2-dimensional bodies in the plane $E={\rm span}\{x_1,x_2\}$,
$$
\bar{K}=\left\{\frac{t_1 x_1}{\|x_1\|_K}+\frac{t_2 x_2}{\|x_2\|_K}: t_1,t_2\ge 0, t_1^p +t_2^p \le 1\right\}
$$
and

 $$
\bar{L}=\left\{x\in\R^n: \|x\|_{\bar L} = \left[\int_0^{\|x\|_{\bar K}^{-1}} f(rx)r^{k-1}dr \right]^{-\frac 1k}\le 1 \right\}.
$$

One can see that the boundary of $\bar K$ consists of  a $p$-arc connecting the points  $\frac{x_1}{\|x_1\|_K}$ and $\frac{x_2}{\|x_2\|_K}$, and two straight line segments connecting the origin with these two points.  Clearly $\bar{K}$ is $p$-convex and $\bar{K}\subset K$. Also note that $\|x_i\|_{\bar K}=\|x_i\|_K$ for $i=1,2$, since $\frac{x_1}{\|x_1\|_K}$ and $\frac{x_2}{\|x_2\|_K}$ are on the boundary of $\bar{K}$, and $\|x_3\|_{\bar K} \ge \|x_3\|_K$ since $\bar{K}\subset K$. It follows that
$\|x_i\|_{\bar L}=\|x_i\|_L \,\, (i=1,2) , $ and $\|x_3\|_{\bar L} \ge \|x_3\|_L$.

 Consider the point $y = \frac{\|x_1\|_{\bar L}}{\|x_1\|_{\bar K}}x_1 + \frac{\|x_2\|_{\bar L}}{\|x_2\|_{\bar K}}x_2$ in the plane $E$. The point $\frac{y}{\|y\|_{\bar K}}$ lies on the $p$-arc connecting $\frac{x_1}{\|x_1\|_{\bar K}}$ and $\frac{x_2}{\|x_2\|_{\bar K}}$. Consider the tangent line to this arc at the point $\frac{y}{\|y\|_{\bar K}}$. This line intersects the segments $[0, {x_i}/{\|x_i\|_{\bar K}}]$, $i=1,2$, at some points
 $\frac{t_ix_i}{\|x_i\|_{\bar K}}$ with $t_i\in(0,1)$.

Since $\frac{t_1x_1}{\|x_1\|_{\bar K}}$, $\frac{t_2x_2}{\|x_2\|_{\bar K}}$ and $\frac{y}{\|y\|_{\bar K}}$ are on the same line, it follows that the
coefficients of $\frac{t_1x_1}{\|x_1\|_{\bar K}}$ and $\frac{t_2x_2}{\|x_2\|_{\bar K}}$ in the equality
$$
\frac{y}{\|y\|_{\bar K}} = \frac{1}{\|y\|_{\bar K}} \( \frac{\|x_1\|_{\bar L}}{t_1}\cdot\frac{t_1x_1}{\|x_1\|_{\bar K}} + \frac{\|x_2\|_{\bar L}}{t_2}\cdot\frac{t_2x_2}{\|x_2\|_{\bar K}} \)
$$
have to add up to 1. Therefore,
$$
\|y\|_{\bar K} = \frac{\|x_1\|_{\bar L}}{t_1} + \frac{\|x_2\|_{\bar L}}{t_2}.
$$
Note also that the line between $\frac{t_1x_1}{\|x_1\|_{\bar K}}$ and $\frac{t_2x_2}{\|x_2\|_{\bar K}}$ separates $\frac{x_3}{\|x_3\|_{\bar K}}$ from the origin, which means that the three points $\frac{t_1x_1}{\|x_1\|_{\bar K}}$, $\frac{t_2x_2}{\|x_2\|_{\bar K}}$ and $\frac{x_3}{\|x_3\|_{\bar K}}$ are in the ``convex position". Applying Ball's theorem on log-concave functions (Thm \ref{th:ball}) to these three points, we have
$$
\[\int_0^{\frac{1}{\|x_3\|_{\bar K}}} f(rx_3)r^{k-1}dr \]^{-\frac 1k}
\le \[\int_0^{\frac{t_1}{\|x_1\|_{\bar K}}} f(rx_1)r^{k-1}dr \]^{-\frac 1k}
+ \[\int_0^{\frac{t_2}{\|x_2\|_{\bar K}}} f(rx_2)r^{k-1}dr \]^{-\frac 1k}.
$$
If we let $s_i = \|x_i\|_{\bar L} \[\int_0^{\frac{t_i}{\|x_i\|_{\bar K}}} f(rx_i)r^{k-1}dr \]^{\frac 1k}$ for each $i=1,2$, the above inequality becomes $$\|x_3\|_{\bar L} \le \frac{\|x_1\|_{\bar L}}{s_1} + \frac{\|x_2\|_{\bar L}}{s_2}.$$ By a change of variables, we get
\begin{eqnarray*}
s_i = t_i \|x_i\|_{\bar L} \[\int_0^{\frac{1}{\|x_i\|_{\bar K}}} f(t_i rx_i)r^{k-1}dr \]^{\frac 1k}
\ge t_i \|x_i\|_{\bar L} \[\int_0^{\frac{1}{\|x_i\|_{\bar K}}} f(rx_i)r^{k-1}dr \]^{\frac 1k}
= t_i
\end{eqnarray*}
for each $i=1,2$. The above inequality comes from the fact that an even log-concave function has to be non-increasing on $[0,\infty)$. Indeed,
\begin{eqnarray*}
f(t_i rx_i) = f\(\frac{1+t_i}{2}\cdot rx_i - \frac{1-t_i}{2}\cdot rx_i\)
\ge f(rx_i)^{\frac{1+t_i}{2}}f(-rx_i)^{\frac{1-t_i}{2}} = f(rx_i).
\end{eqnarray*}
Putting all together, we have
\begin{eqnarray*}
  \|x_3\|_L \le \|x_3\|_{\bar L}\le \frac{\|x_1\|_{\bar L}}{s_1} + \frac{\|x_2\|_{\bar L}}{s_2}
\le \frac{\|x_1\|_{\bar L}}{t_1} + \frac{\|x_2\|_{\bar L}}{t_2} = \|y\|_{\bar K}.
\end{eqnarray*}
Using the  $p$-convexity of $\bar{K}$, we have
\begin{eqnarray*}
\|y\|_{\bar K}^p \le  \left\|\frac{\|x_1\|_{\bar L}}{\|x_1\|_{\bar K}}x_1\right\|_{\bar K}^p + \left\|\frac{\|x_2\|_{\bar L}}{\|x_2\|_{\bar K}}x_2\right\|_{\bar K}^p
=\|x_1\|_{\bar L}^p + \|x_2\|_{\bar L}^p = \|x_1\|_L^p + \|x_2\|_L^p,
\end{eqnarray*}
and therefore $\|x_3\|_L^p \le \|x_1\|_L^p + \|x_2\|_L^p$.

\end{proof}

\begin{corollary}
Let $\mu$ be a symmetric log-concave measure and $K$ a symmetric $p$-convex body in $\R^n$ for $p \in (0,1]$. Then the intersection body $I_\mu K$ of $K$ is $q$-convex with $q =\left[(1/p-1)(n-1) + 1 \right]^{-1}$.
\end{corollary}

\begin{proof}
Let $f$ be the density function of $\mu$. By Theorem \ref{th:log-concave}, the body $L$ with the Minkowski functional
$$
\|x\|_L = \left[(n-1)\int_0^{\|x\|_K^{-1}} f(rx)r^{n-2}dr\right]^{\frac{-1}{n-1}},\quad x\in\R^n,
$$
is $p$-convex.

On the other hand, the intersection body $I_\mu K$ of $K$ is given by the radial function
\begin{eqnarray*}
\rho_{I_\mu K}(u) &=& \mu_{n-1}(K\cap u^\bot)  \\
&=& \int_{\R^n} 1_{K\cap u^\bot}(x)f(x) dx \quad=\quad \int_{S^{n-1}\cap u^\bot} \int_0^{\|u\|_K^{-1}}f(rv)r^{n-2} drdv \\
&=& \frac{1}{n-1}\int_{S^{n-1}\cap u^\bot} \|v\|_L^{-n+1}dv \quad=\quad |L\cap u^\bot|_{n-1} \\ &=& \rho_{IL}(u),
\end{eqnarray*}
which means $I_\mu K = IL$. By Theorem \ref{th:pbus}, $IL$ is $q$-convex with $q =\left[ (1/p-1)(n-1) +1 \right]^{-1}$, and therefore so is $I_\mu K$.

\end{proof}

We conclude this section with an example that shows that  the condition on $f$ to be even in Theorem \ref{th:log-concave} cannot be dropped.

\noindent{\bf Example 1}.
Let $\mu$ be a log-concave measure on $\R^n$ with density
$$f(x_1,\ldots,x_n)=\left\{\begin{array}{ll} 1, &\text{if }x_1+x_2 \ge 2^{1-1/p},\\
0, &\text{otherwise.}\end{array}
\right.$$  Consider the $p$-convex body $K=B_p^n$ for $p\in(0,1)$. If $L$ is the body defined in Theorem \ref{th:log-concave}, then  $\|e_1+e_2\|_L=0$ and $\|e_1\|_L=\|e_2\|_L>0$, which means $L$ is not $q$-convex for any $q>0$.\\

\section{Non-symmetric cases and $s$-concave measures}

Note that Ball's theorem (Thm \ref{th:ball}) remains valid even if $f$ is not even, as was shown by Klartag \cite{Kl}. On the other hand,
as we explained above, Theorem \ref{th:log-concave} does not hold for non-symmetric log-concave measures. However, if we restrict ourselves to
the class of $s$-concave measures, $s>0$, then it is possible to give a version of Theorem \ref{th:log-concave} for non-symmetric measures.

Borell \cite{Bor} introduced the classes $\mathfrak{M}_s(\Omega)$, ($-\infty
\le s \le \infty$, $\Omega \subset \R^n$ open convex) of $s$-concave
measures, which are Radon measures $\mu$ on $\Omega$ satisfying the
following condition: the inequality
$$
\mu(\la A + (1-\la)B) \ge \left[\la \mu(A)^s + (1-\la)\mu(B)^s
\right]^{\frac 1s}
$$
holds for all nonempty compact $A, B \subset \Omega$ and all $\la \in
(0,1)$. In particular,   $s=0$ gives the class of log-concave measures.

Let us consider the case $0<s<1/n$. According to Borell, $\mu$ is
$s$-concave if and only if the support of $\mu$ is $n$-dimensional
and $d\mu = fdm$ for some $f \in L^1_{loc}(\Omega)$ such that
$f^{\frac{s}{1-ns}}$ is a concave function on $\Omega$.

\begin{theorem}
Let $\mu$ be an $s$-concave measure on $\Omega\subset\R^n$ with density $f$, for $0<s<1/n$, and $K$ a $p$-convex body in $\Omega$, for $p\in (0,1]$. If $k\ge 1$, then the body $L$ whose Minkowski functional is given by
$$
\|x\|_L = \left[\int_0^\infty 1_K(rx)f(rx)r^{k-1}dr\right]^{-\frac 1k},\quad x\in\R^n
$$
is $q$-convex with $q= \[\big(\frac 1p-1\big)\big(\frac 1s-n\big)\frac 1k + \frac 1p\]^{-1}$.
\end{theorem}

\begin{proof}
Let $x_1,x_2 \in \R^{n}$ and $x_3=x_1+x_2$. Then, for $i=1,2$,
\begin{eqnarray*}
\|x_i\|_L^{-k} &=& \int_0^\infty 1_K(rx_i) f(rx_i)r^{k-1}dr =
\frac 1p\int_0^\infty 1_K(s^{-\frac 1p}x_i) f(s^{-\frac 1p}x_i) s^{-\frac kp-1}ds \\
&=& \frac 1p\int_0^\infty F_i(s)ds =
\frac 1p\int_0^\infty \Big|\{s \in (0,\infty) : F_i(s)>t\}\Big|\, dt,
\end{eqnarray*}
where $F_i(s) = 1_K(s^{-\frac 1p}x_i) f(s^{-\frac 1p}x_i)
s^{-\frac kp-1}$ for each $i=1,2,3$.  We claim that
$$
2^{\frac kq+1} F_3(s_3) \ge F_1(s_1)^{\la_1}F_2(s_2)^{\la_2}
$$
whenever $s_3 = s_1 + s_2$ and $\la_i = \frac{s_i}{s_1+s_2}$ for
$i=1,2$. Indeed, since
\begin{eqnarray*}
s_3^{-\frac 1p}x_3 &=& \left(\frac{s_1}{s_1+s_2}\right)^{\frac
1p}s_1^{-\frac 1p}x_1 + \left(\frac{s_2}{s_1+s_2}\right)^{\frac
1p}s_2^{-\frac 1p}x_2 \\
&=& \la_1(\la_1^{\frac 1p -1}s_1^{-\frac 1p}x_1) +
\la_2(\la_2^{\frac 1p -1}s_2^{-\frac 1p}x_2),
\end{eqnarray*}
the concavity of $f^\ga$, $\ga=\frac{s}{1-ns}$, gives
\begin{eqnarray*}
f^\ga(s_3^{-\frac 1p}x_3) &\ge& \la_1 f^\ga(\la_1^{\frac 1p-1}s_1^{-\frac 1p}x_1)
+ \la_2 f^\ga(\la_2^{\frac 1p -1}s_2^{-\frac 1p}x_2) \\
&\ge& \[ f^\ga(\la_1^{\frac 1p -1}s_1^{-\frac 1p}x_1)\]^{\la_1}
\[f^\ga(\la_2^{\frac 1p -1}s_2^{-\frac 1p}x_2) \]^{\la_2}\\
&\ge& \[\la_1^{\frac 1p -1} f^\ga(s_1^{-\frac 1p}x_1)\]^{\la_1}
\[\la_2^{\frac 1p -1}f^\ga(s_2^{-\frac 1p}x_2) \]^{\la_2}\\
&=& \( \[\(\frac{s_1}{s_3}\)^{\frac 1\ga(\frac 1p-1)} f(s_1^{-\frac 1p}x_1) \]^{\la_1}
\[\(\frac{s_2}{s_3}\)^{\frac 1\ga(\frac 1p -1)} f(s_2^{-\frac 1p}x_2) \]^{\la_2} \)^\ga ,
\end{eqnarray*}
that is,
$$
s_3^{\frac 1\ga(\frac 1p -1)} f(s_3^{-\frac 1p}x_3) \ge
\prod_{i=1}^2 \left[ s_i^{\frac 1\ga(\frac 1p -1)} f(s_i^{-\frac 1p}x_i) \right]^{\la_i}.
$$
On the other hand, note that
$$
\frac{2}{s_3} = \frac{\la_1}{s_1} + \frac{\la_2}{s_2}
\ge \(\frac{1}{s_1}\)^{\la_1} \(\frac{1}{s_2}\)^{\la_2}
$$ and
$$
1_K(s_3^{-\frac 1p}x_3) \ge 1_K(s_1^{-\frac 1p}x_1) 1_K(s_2^{-\frac 1p}x_2),
$$
since $s_3^{-\frac 1p}x_3 = \la_1^{\frac 1p}(s_1^{-\frac 1p}x_1) + \la_2^{\frac 1p
}(s_2^{-\frac 1p}x_2)$. Thus
\begin{eqnarray*}
F_1(s_1)^{\la_1} F_2(s_2)^{\la_2}
&=& \prod_{i=1}^2 \[1_K(s_i^{-\frac 1p}x_i) f(s_i^{-\frac 1p}x_i)s_i^{-\frac kp-1} \]^{\la_i} \\
&\le& 1_K(s_3^{-\frac 1p}x_3) \prod_{i=1}^2 \[ s_i^{\frac 1\ga(\frac 1p-1)}
f(s_i^{-\frac 1p}x_i) \cdot \(\frac{1}{s_i}\)^{\frac 1\ga (\frac 1p-1) + \frac kp+1}\]^{\la_i} \\
&\le& 1_K(s_3^{-\frac 1p}x_3) f(s_3^{-\frac 1p}x_3)
2^{\frac 1\ga (\frac 1p-1)+\frac kp+1} s_3^{-\frac kp-1} \\
&\le& 2^{\frac kq +1} F_3(s_3).
\end{eqnarray*}
It follows that for every $t>0$
$$
\{s_3 : 2^{\frac kq +1} F_3(s_3)>t\}  \supset \{s_1 : F_1(s_1)>t\}+ \{s_2 : F_2(s_2)>t\}.
$$
Applying the Brunn-Minkowski inequality, we have
\begin{eqnarray*}
\|x_1+x_2\|_L^{-k} &=& \|x_3\|_L^{-k} = \frac 1p\int_0^\infty F_3(s)ds \\
&=& \frac{1}{2^{\frac kq +1}} \cdot \frac 1p\int_0^\infty \Big|\{s_3 \in
(0,\infty) : 2^{\frac kq +1}F_3(s_3)>t\}\Big|\, dt \\
&\ge& \frac{1}{2^{\frac kq+1}} \cdot \frac 1p\int_0^\infty
\(\Big|\{s_1 : F_1(s_1)>t\}\Big| + \Big|\{s_2 :F_2(s_2)>t\}\Big| \) dt \\
&=& \frac{1}{2^{\frac kq+1}} (\|x_1\|_L^{-k} + \|x_2\|_L^{-k}).
\end{eqnarray*}
Thus,
\begin{eqnarray*}
\|x_1+x_2\|_L &\le& 2^{\frac 1q} \(\frac{\|x_1\|_L^{-k} + \|x_2\|_L^{-k}}{2}\)^{-\frac 1k}
= \[ \frac 12 \( \frac{(\|x_1\|_L^{-q})^{\frac kq} + (\|x_1\|_L^{-q})^{\frac kq}}{2} \)^{\frac qk}\]^{-\frac 1q} \\
&\le& \[ \frac 12 \( \frac{\|x_1\|_L^{-q} + \|x_2\|_L^{-q} }{2} \)
\]^{-\frac 1q}
\le \[ \frac 12 \( \frac{\|x_1\|_L^q + \|x_2\|_L^q}{2} \)^{-1} \]^{-\frac 1q}\\
&=& \( \|x_1\|_L^q + \|x_2\|_L^q \)^{\frac 1q},
\end{eqnarray*}
which means that $L$ is $q$-convex.

\end{proof}

The following example shows that the value of $q$ in the above theorem is sharp.\\

\noindent{\bf Example 2.}
Let $\mu$ be an $s$-concave measure on $\Omega=\{(x_1,\ldots,x_n)\in\R^n: x_1\ge 0\}$ for $s>0$ with density $$f(x_1,\ldots,x_n)=|x_1|^{1/s-n}$$
and let
$$
K=\left\{(x_1,\ldots,x_n):  x_1 \ge 0, \left|\frac{x_1+x_2}{2}\right|^p+\left|\frac{x_1-x_2}{2}\right|^p\le 1, \,\, |x_i|\le 1 \,\,\forall i=3,\ldots,n\right\}.
$$
Note that $\|e_1\|_K = 2^{1-1/p}$ and $\|e_1+e_2\|_K=\|e_1-e_2\|_K=1$. If $L$ is the body defined by $K$ in the above theorem, then
$$
\|e_1\|_L = \[\int_0^{2^{1-1/p}}r^{1/s-n}r^{k-1}dr\]^{-\frac 1k}= \[\frac{2^{\(1-\frac 1p\)\(\frac 1s-n+k\)}}{\frac 1s-n+k}\]^{-\frac1k}
$$
and
$$
\|e_1+e_2\|_L = \[\int_0^1 r^{1/s-n}r^{k-1}dr\]^{-\frac 1k}= \[\frac 1s-n+k\]^{\frac 1k}.
$$
If $L$ is $q$-convex for some $q$, then the inequality $\|2e_1\|_L \le \(\|e_1+e_2\|_L^q +\|e_1-e_2\|_L^q\)^{1/q}$ implies
$$
2\[\frac{2^{\(1-\frac 1p\)\(\frac 1s-n+k\)}}{\frac 1s-n+k}\]^{-\frac 1k} \le 2^{\frac 1q}\[\frac 1s-n+k\]^{\frac 1k}
$$
that is,
$$
q\le \[\(\frac 1p-1\)\(\frac 1s-n\)\frac 1k+\frac 1p\]^{-1}.
$$

Note that in our construction $\Omega$ is not open, as opposed to what
we said in the beginning of Section 4. This is done for the sake of simplicity of the
presentation. To be more precise one would need to define
$\Omega=\{(x_1,\ldots,x_n)\in\R^n: x_1>-\epsilon\}$ and
$f(x_1,\ldots,x_n)=|x_1+\epsilon|^{1/s-n}$, for $\epsilon>0$, and then
send $\epsilon \to 0^+$.

 \end{document}